\documentclass[english]{amsart}
\usepackage[T1]{fontenc}
\usepackage{euscript}
\usepackage{textcomp}
\usepackage{subfig}
\usepackage{amsmath}
\usepackage{amssymb}
\usepackage{amsthm}
\usepackage[english]{babel}
\usepackage{graphicx}
\theoremstyle{definition}
\newtheorem{definition}{Definition}
\newtheorem{theorem}{Theorem}
\newtheorem{lemma}[theorem]{Lemma}

\newtheorem{remark}[theorem]{Remark}

\title{On the Hausdorff dimension of the Rauzy gasket}
\author {Artur Avila}
\address{CNRS UMR 7586, Institut de Math\'ematiques de Jussieu - Paris Rive Gauche, Batiment Sophie Germain, Case 7021, 75205 Paris Cedex 13, France\\
and IMPA, Estrada Dona Castarino 110, 22460-320, Rio de Janeiro, Brazil}
\email{artur@math.jussieu.fr}
\author{Pascal Hubert}
\address{Institut de Math\'ematiques de Marseille, 39 rue F. Joliot-Curie, 13453 Marseille Cedex 20, France}
\email{pascal.hubert@univ-amu.fr}
\author {Alexandra Skripchenko}
\address{Faculty of Mathematics, National Research University Higher School of Economics, Vavilova St. 7, 112312 Moscow, Russia}
\email{sashaskrip@gmail.com}

\begin{document}
\maketitle
\begin{abstract} In this paper, we prove that the Hausdorff dimension of the Rauzy gasket is less than 2. By this result, we answer a question addressed by Pierre Arnoux. Also, this question is a very particular case of the conjecture stated by S.P. Novikov and A. Maltsev in 2003.  
\end{abstract}

\section{Introduction}
The Rauzy gasket (see Figure \ref{RaGu}) was described for the first time by G. Levitt in 1993 in \cite{L} and was associated with the simplest example of pseudogroups of rotations. Later the Rauzy gasket was studied by I. Dynnikov and R. De Leo (see \cite{DD}) in connection with Novikov's problem (\cite{N}) of plane sections of triply periodic surfaces. In \cite{AS} independently P. Arnoux and S. Starosta  reintroduced this object as the subset of standard 2-dimensional simplex associated with letter frequencies of ternary episturmian words. The name Rauzy gasket was used for the first time in \cite{AS}. 

The same fractal appears in connection with systems of isometries of thin type that are described by 2 independent parameters. A detailed description of the last approach is provided in the next section. In all these cases, the Rauzy gasket plays the role of a parameter space endowed with a dynamics by piecewise projective maps.

It was proved by Levitt and J.-C. Yoccoz in \cite{L}, Arnoux and Starosta in \cite{AS} and by Dynnikov and De Leo ( \cite{DD}) by different techniques that the Rauzy gasket has zero Lebesgue measure (see, in particular, \cite{MN} for the main approach used by Arnoux and Starosta to prove their result).

Hausdorff dimension of the Rauzy gasket was estimated numerically in \cite{DD} ($1.7$ and $1.8$ were suggested as lower and upper bounds). However, there were no theoretical estimates. Arnoux asked whether this Hausdorff dimension is less than 2 or equal to 2 (see also \cite{AS} for other interesting open questions). The same problem but for much more general situation was posed by A. Maltsev and S.P. Novikov in \cite{NM}.
In this paper we prove:
\begin{theorem} \label{thm:main}
The Hausdorff dimension of the Rauzy gasket is less than 2. 
\end{theorem}

\begin{remark}
Our statement also proves the conjecture about the Hausdorff dimension of the set of  chaotic regimes formulated by A. Maltsev and S.P. Novikov that we mentioned above for a very particular family of surfaces. This  result follows directly from our theorem and the construction in \cite{DD}.
\end{remark}

\begin{remark} An upper bound can be deduced from the proof of Theorem \ref{thm:main} but it would be much weaker than the known numerical estimates. 
\end{remark}

\begin{figure}[h]
\centering
\vspace{-20cm}
\hspace{-5cm}
\includegraphics[scale=1.1]{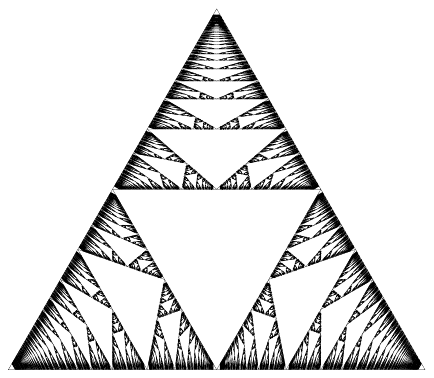}
\vspace{-3cm}
\caption{The Rauzy Gasket}
\label{RaGu}
\end{figure}

\subsection{Organization of the paper} 

The paper is organized as follows. 

In section \ref{Def} we recall the definition of systems of isometries and describe the particular family we work with in the current paper. 

In section \ref{RI} we describe the Rauzy induction for systems of isometries and the related symbolic dynamics. In particular, we define the Rauzy gasket in terms of systems of isometries and describe the corresponding Markov map and Markov partition. 

Section \ref{Co} is dedicated to the cocycle associated with the induction. Like in case of IET, the definition requires a suspension construction that is also presented in the same section. This cocycle will be used later for the construction of the suspension flow. 

In section \ref{UE} we prove that the Markov map is uniformly expanding in a sense of \cite{AGY}.

In \ref{DE} and in \ref{ET} we provide some distortion estimates for the cocycle that are based on so called Kerckhoff lemma (see Appendix A in \cite{AGY} and Theorem  4.2 in \cite{AR}).

Section \ref{RF} is about the roof function and the suspension flow: we construct the roof function associated with the cocycle and use this roof function to define the flow. 

In section \ref{ETRF} we prove that the roof function has exponential tails (the idea of the proof comes from Theorem 4.6 in \cite{AGY}).

The proof of Theorem \ref{thm:main} in presented in section \ref{HD}. First, we use the exponential tails of the roof function to check that the Markov map is fast decaying in the sense of \cite{AD}.  So, one can verify that Theorem 26 in \cite{AD} is applicable in our case and it implies the main result.

In the last section \ref{MD} we briefly explain how to extend our result for a multi-dimensional version of the Rauzy gasket. 

\subsection{Acknowledgments} We heartily thank P. Arnoux who gave inspiring talks on the Rauzy gasket and asked the question to the authors. We thank A. Bufetov for interesting discussions on distorsion estimates obtained in \cite{B}. We also thank V. Delecroix for many useful explanations and I. Dynnikov and the unknown referee for several improvements of the first version of the paper. 

The work was partially supported by the projet ANR blanc GeoDyM, by the ERC Starting Grant \textquotedblleft Quasiperiodic\textquotedblright and by the Balzan project of Jacob Palis. The third author was also supported by the Fondation Sciences Math\'ematiques de Paris. 

\section{Definition}\label{Def}
The notion of systems of isometries was introduced by G. Levitt, D. Gaboriau and F. Paulin in \cite {GLP}.
\begin{definition}
A {system of isometries} $S$ consists of a finite disjoint union $D$ of compact subintervals of the real line $\mathbb R$ (\emph{support multi-interval}) together with a finite number $n$ of partially defined orientation preserving isometries $\phi_{j}: A_{j} \to B_{j}$, where each base of $A_{j}, B_{j}$ is a compact subinterval of $D$. 
\end{definition}

See, for example, Fig. \ref{S}.
\begin{figure}[h]
\centering
\vspace{-4cm}
\includegraphics[height=10cm, width=8cm]{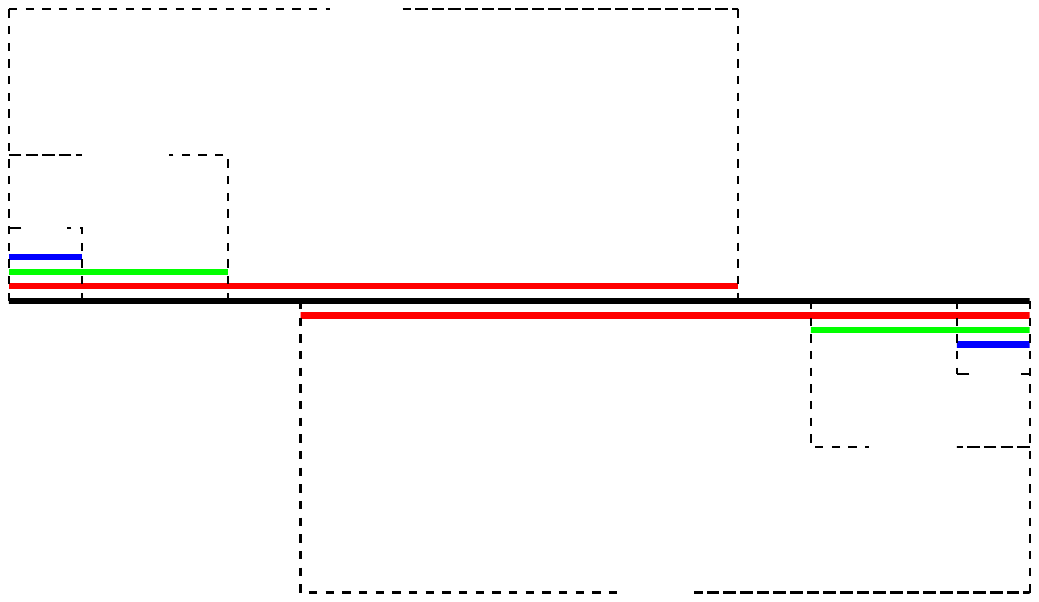}
\put(-142,173){$A_{1}$}
\put(-91,72){$B_{1}$}
\put(-187,147){$A_{2}$}
\put(-45,97){$B_{2}$}
\put(-200,138){$A_{3}$}
\put(-30,108){$B_{3}$}
\vspace{-2.5cm}
\caption{System of isometries}
\label{S}
\end{figure}
Systems of isometries can be considered as a generalization of IET and interval translation mappings (ITM), so it is natural to define the orbit of such system in the same way as it was done for IET.
\begin{definition}
Two points $x,y$ in $D$ belong to the same $S$-orbit if there exists a word $\phi$ that consists of $\phi_{i}$ and inverse to them such that $\phi(x)=y$.
\end{definition}

We denote the orbit of point $x$ by $\Gamma_{x}(S)$. 

Now one can define the equivalence relationship on systems of isometries. Informally, two systems of isometries with the same behavior of orbits are called equivalent. The formal definition we use comes from \cite{D} and is given for $D=[A,B]$ (which is always the case in the current paper).

\begin{definition}
Two systems of isometries $S$ and $S^{'}$ with support intervals $[A,B]$ and $[A^{'},B^{'}]$, respectively, are called \emph{equivalent}, if there is a real number $t \in {\mathbb R}$ and an interval $[A_{0},B_{0}] \subset [A,B] \cap  [A^{'}+t,B^{'}+t]$ such that
\begin {enumerate}
\item every orbit of each of the systems $S$ and $S^{'}+t$ contains a point lying in $[A_{0},B_{0}]$
\item for each point $x\in [A_{0},B_{0}]$ the graphs $\Gamma_{x}(S)$ and $\Gamma_{x}(S^{'})$ are homotopy equivalent through mappings that are identical on $[A_{0},B_{0}]$ and such that the full preimage of each vertex contains only finitely many vertices of the other graph.
\end {enumerate}
\end{definition}

\noindent One can check that it is an equivalence relation.

In the current paper we concentrate on a particular class of systems of isometries.
\begin{definition}\label{def1}
A system of isometries $S$ is called \emph{special} if the following restrictions hold:
\begin{itemize}
\item $D$ is an interval of the real line, say, $[0,1]$;
\item $n=3$;
\item all $A_{i}$ start in $0$;
\item all $B_{i}$ end in $1$;
\item $\Sigma_{i=1}^{3}|A_{i}|=1,$ where $|A|$ means length of subinterval $A$;
\item $|A_{1}|>|A_{2}|>|A_{3}|.$
\end{itemize}
\end{definition}

So, any special system $S$ can be described in the following way:
\begin{equation}
\begin{split}
\label{100}
S=([0,a+b+c];& [0,a]\leftrightarrow[b+c,a+b+c],\\
                    & [0,b]\leftrightarrow[a+c,a+b+c], \\
                    & [0,c]\leftrightarrow[a+b,a+b+c])
\end{split}
\end{equation}
with $a>b>c>0$, $a+b+c=1$.
  
We are only interested in the most generic case of special
system of isometries in the sense that no integral linear relation holds for the parameters a, b, c except those that must hold by definition.
 
We work with special systems of isometries of \emph{thin type}. By the latter we mean a system of isometries for which an equivalent system may have arbitrarily small support (or, equivalently, all orbits are everywhere dense). Thin type was discovered by Levitt in \cite{L} and sometimes is mentioned as \emph{Levitt} or \emph{exotic} case. In \cite{D} Dynnikov showed a strategy how to construct a symmetric 3-periodic surface whose intersections with a family of planes of fixed direction have chaotic behavior using a system of isometries of thin type.

\section {The Rauzy induction and symbolic dynamics} \label{RI}
\subsection{The Rauzy induction without acceleration}
In the theory of IET the Rauzy induction is a Euclid type algorithm that transforms an original IET into another one operating on a smaller interval but equivalent from the point of view of the topology of the corresponding measured foliation. Its iteration can be viewed as a generalized version of continued fraction expansion. This process can also be considered as a variation of the Rips machine algorithm for band complexes in the theory of {$\mathbb R$}-trees (\cite{GLP}).  

The modification of the Rauzy induction algorithm for systems of isometries was introduced by Dynnikov in \cite{D}. The main idea is that from any system of isometries one constructs a sequence of systems of isometries equivalent  to the original one but with a smaller support. Combinatorial properties of this sequence are responsible for ``ergodic" properties of the original system of isometries. 

\emph{The Rauzy induction} for a special system of isometries is a recursive application of admissible transmissions followed by reductions as described below.

\begin{definition}
Let $$S=([0,a+b+c];[0,a]\leftrightarrow [b+c,a+b+c];$$
$$ [0,b] \leftrightarrow [a+c,a+b+c]; [0,c]\leftrightarrow [a+b, a+b+c])$$
be a special system of isometries. So, two of the subintervals,
$\left[a+c,a+b+c\right]$ and $\left[a+b,a+b+c\right]$, say, are contained in the third one $\left[b+c,a+b+c\right]$, say. Let $S^{'}$ be the system of isometries obtained from $S$ by replacing the pair $\left[0,b\right]\leftrightarrow\left[a+c,a+b+c\right]$
by the pair $\left[0,b\right]\leftrightarrow\left[a-b,a\right]$ and the pair $\left[0,c\right]\leftrightarrow\left[a+b,a+b+c\right]$ by the pair $\left[0,c\right]\leftrightarrow\left[a-c,a\right]$

We say that $S^{'}$is obtained from $S$ by a \emph{transmission} (on the right).
\end {definition}

\begin{definition}
Let
$$
S=\left(\left[A,B\right];\left[a_{1},b_{1}\right]\leftrightarrow\left[c_{1},d_{1}\right];\left[a_{2},b_{2}\right]\leftrightarrow\left[c_{2},d_{2}\right];\left[a_{3},b_{3}\right]\leftrightarrow\left[c_{3},d_{3}\right]\right)
$$
be a system of isometries (not necessarily special) and let $d_{1}=B$. We call all
endpoints of our subintervals \emph{critical points}. Assume that
the point $B$ is not covered by any interval from $S$ except $[c_1,d_{1}]$
and that the interior of the interval $\left[c_{1},d_{1}\right]$ contains
a critical point. Let $u$ the rightmost such point. Then the interval $[u,B]$ is covered only by one interval from our system. Replacing the pair $\left[a_{1},b_{1}\right]\leftrightarrow\left[c_{1},d_{1}\right]$
with $\left[a_{1},b_{1}-d_{1}+u\right]\leftrightarrow$$\left[c_{1},u\right]$
in $S$ with simultaneous cutting off the part $[u,B]$ from the support interval will be called a\emph{ reduction on the right} (of the pair $\left[a_{1},b_{1}\right]\leftrightarrow\left[c_{1},d_{1}\right]$). 
\end{definition}

Note that application of the Rauzy induction to a special system of isometries gives us a special system of isometries again (see Fig. \ref{1}). 
The pair of subintervals that was reduced is called a \emph{winner} (like in a case of IET). 
\begin{figure}[h]
\includegraphics[width=9cm,height=11cm]{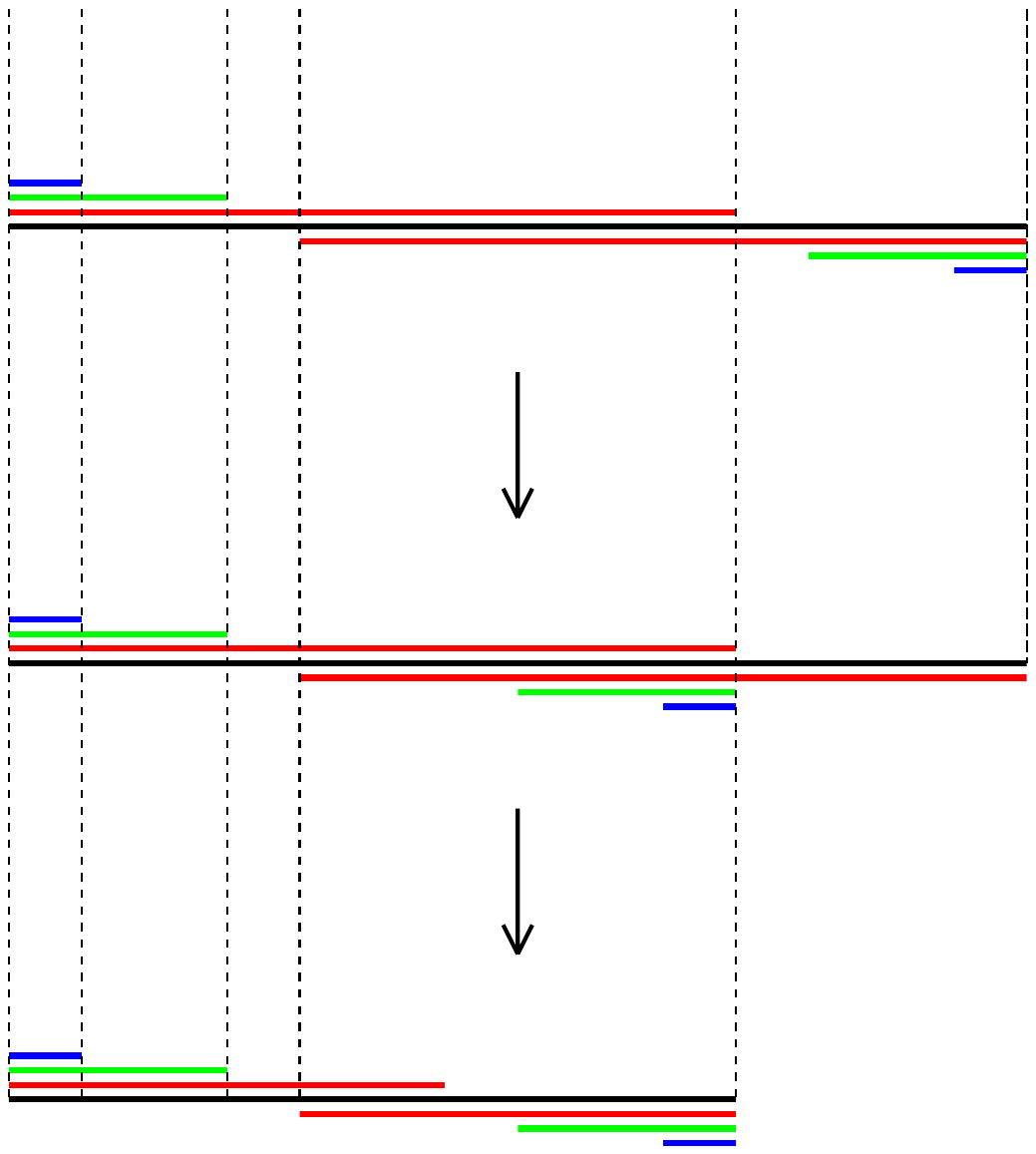}
\put(-214,265){$c$}
\put(-230,265){$0$}
\put(-185,265){$b$}
\put(-178,265){$b+c$}
\put(-83,265){$a$}
\put(-40,265){$a+b+c$}
\vspace{-1cm}
\caption{The Rauzy induction: transmission of $b$ and $c$ intervals and reduction of $a$-interval.}
\label{1}
\end{figure} 

Following \cite{D}, we have the following obvious 
\begin{lemma}
The Rauzy induction does not influence on the existence of the finite orbits or on their property to be everywhere dense: the origin and the image are equivalent.
\end{lemma} 

We say that a system of isometries has a \emph{hole}
if there are some points in the support interval that are not covered by
any interval from $S$. This means in particular that our system has points
with finite orbits. Therefore, one can stop the Rauzy induction once it results in a system with a hole.

One can check that a system of isometries of thin type is exactly such a system for which the Rauzy induction can be applied for infinite number of times, and we never get a hole (see \cite{S12} for details).

\subsection{The Rauzy gasket}
Let $\EuScript{A} =\{ 1,2,3 \}.$ We enumerate 3 subintervals (from the biggest to the smallest) and check what happens with them during the Rauzy induction. 
Since in accordance with definition \ref{def1} the intervals are always enumerated from the biggest to the smallest, we sometimes need to change this enumeration after a step of the Rauzy induction. The configuration then can be coded by permutations of three elements. 

So, like in case of IET, for each special system of isometries we associate data of two types: a collection of three lengths ($a,b,c$) and the order of the subintervals with respect to the original one. Thus, the parameter space $\EuScript{V}= \mathbb R^{3}\times S_{3},$ with the normalizing restriction $a+b+c=1.$ 

One step of the Rauzy induction can be coded by one of the following data collection (we express the old length of subintervals in terms of the new one):
\begin{itemize}
\item Matrix 
$$R_{1}=\begin{pmatrix}
1 & 1 & 1\\
0 & 1 & 0 \\
0 & 0 & 1
\end{pmatrix},$$
permutation $(1,2,3)$;
\item Matrix $$R_{2}=\begin{pmatrix}
1 & 1 & 1\\
1 & 0 & 0 \\
0 & 0 & 1
\end{pmatrix},$$
permutation $(2,1,3)$;
\item Matrix $$R_{3}=\begin{pmatrix}
1 & 1 & 1\\
1 & 0 & 0 \\
0 & 1 & 0
\end{pmatrix},$$
permutation $(3,1,2).$
\end{itemize}

Comparing formulas for the Rauzy induction with the maps that appear in \cite{AS} in description of the Rauzy gasket as an iterated function system, it is easy to see that the following statement holds:
\begin{lemma}
The set of parameters $(a,b,c)$, such that corresponding special systems of isometries are of thin type, forms the Rauzy gasket.
\end{lemma}
 
In the current paper we use this property as a definition of the Rauzy gasket. 

\subsection{The Accelerated Rauzy induction} \label{Rn}

One can construct an \emph{accelerated} version of the Rauzy induction. We define a \emph{generalized iteration} of the Rauzy induction by analogy with a step of the fast version of Euclid's algorithm, which involves the division with remainder instead of subtraction of the smaller number from the larger. It may happen that only one of the three pairs of intervals is subject to reduction in several consecutive steps of the Rauzy induction (and the intervals from the second and the third pair are involved only in transmissions). It means that there is the same winner for several consecutive steps of the algorithm.
In this case we consider the result of such a sequence of the Rauzy induction iterations as the result of applying of one generalized iteration. This kind of acceleration for IET was described by Zorich in \cite{Z}. 

\noindent The matrix $R(n)$  of the one step of the accelerated Rauzy induction is the following:
$$\begin{pmatrix}
n & 1 & n\\
1 & 0 & 0 \\
0 & 0 & 1
\end{pmatrix}$$
or 
$$\begin{pmatrix}
n & n & 1\\
0 & 0 & 1 \\
0 & 1 & 0
\end{pmatrix},$$
where $n$ is a number of simple Rauzy inductions included in one generalized iteration.
There is an evident 
\begin{lemma}
$R_{n}$ is the result of $n-1$ applications of $R_{1}$ and one application of $R_{2}$ or $R_{3}$:
$R(n)=R_{1}\cdot\cdots\cdot R_{1}\cdot R_{2}$ or $R(n)=R_{1}\cdot\cdots\cdot R_{1}\cdot R_{3}$
\end{lemma}

\subsection{The Markov Map}\label{mm}
In the case of special systems of isometries the parameter space $X$ is the triangle with vertices $(1:0:0)$, $(0:1:0)$, $(0:0:1)$. 
The Rauzy induction defines a partition of $X$ in the following way:
\begin{itemize}
\item on step zero $X$ is divided into four subsimplices: 
\begin{itemize}
\item $X^{0}_{1}$ with vertices $(1:1:0)$, $(1:0:0)$, $(1:0:1)$, it corresponds to the coding $(1,2,3)$ and $(1,3,2)$;
\item $X^{0}_{2}$ with vertices $(0:1:1)$, $(0:1:0)$, $(1:1:0)$, it corresponds to the coding $(2,1,3)$ and $(2,3,1)$;
\item $X^{0}_{2}$ with vertices $(1:0:1)$, $(0:0:1)$, $(0:1:1)$, it corresponds to the coding $(3,1,2)$ and $(3,2,1)$;
\item $X^{0}_{0}$ with vertices $(1:0:1)$, $(0:1:1)$, $(1:1:0)$, it corresponds to the hole; 
\end{itemize}
\item the renormalized version of the induction map for $l=(a,b,c)$ with $a+b+c=1$ is
$T: X\to X, T(l)=\frac{Rl}{||Rl||}$, where $R$ is the matrix of the induction.
\item after one step of the Rauzy induction one of three subsimplices (depending where the point that we examine was located) will be also divided into four parts in the same way etc.
\end{itemize}
We enumerate the steps of the (non-accelerated) induction by lower $n$ and the number of the part of each step by the upper index $i: X^{i}_{n}$ is the cell with the corresponding address.
\begin{lemma}
$T$ is a Markov map, and $(X^{i}_{n})$ is a Markov partition.
\end{lemma}

The Markov partition is shown on Figure \ref{RaGu}; the Rauzy gasket (black part) is a fractal subset of $X$ determined by the systems of isometries of thin type; the white part corresponds to the systems of isometries such that the hole was obtained after some steps of the Rauzy induction. 

\subsection {The Rauzy graph}
Like in case of IET, we use \emph{the Rauzy graph} to describe the combinatorics of the accelerated Rauzy induction. Acceleration means that the combinatorics changes after each step. In this paper we work with minimal systems of isometries, and the vertex representing the hole can be excluded from the graph (we call this exclusion an ''adjustment"). So it is enough to consider the graph on $6$ vertices.
 
Then, the vertices of the adjusted Rauzy graph are all permutations of 3 elements, and 2 vertices are connected by an arrow if and only if there exists a realization of it by the Rauzy induction. For example, looking at one step of the Rauzy induction, we see that there is $(1,2,3)\rightarrow (2,1,3)$ but there is no arrow between $(1,2,3)$ and $(3,2,1).$ The adjusted Rauzy graph for the accelerated Rauzy induction is shown on Figure \ref{G}.

\begin{figure}[h]
\includegraphics[width=9cm,height=11cm]{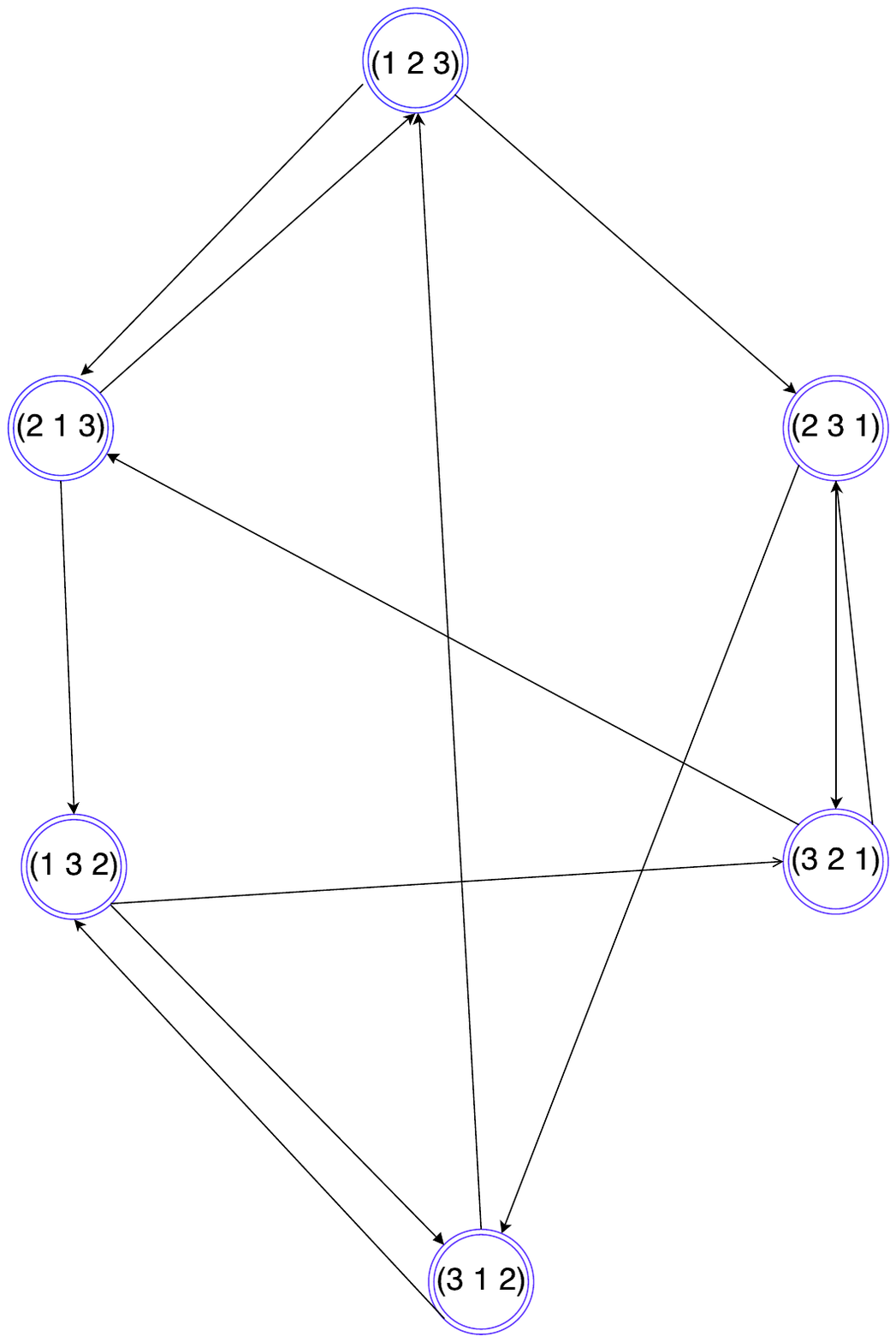}
\vspace{-1cm}
\caption{The adjusted accelerated Rauzy Graph}
\label{G}
\end{figure} 

We have the following obvious 
\begin{lemma}
The Rauzy graph is connected.
\end{lemma}

For future constructions we will also need the following definition (see \cite{AGY}):
\begin{definition}
A path $\gamma$ in the Rauzy graph is called \emph{complete} if every $\alpha \in \EuScript A$ is a winner of some arrow composing $\gamma$.
\end{definition}

\subsection {The Markov shift}
One can also consider the action of the non-accelerated Rauzy induction on the accelerated adjusted Rauzy graph. Then, each vertex of the adjusted Rauzy graph will split into countable number of vertices, and the same happens to the corresponding Markov cell. Each small Markov cell is coded by a permutation (that comes from the coding of vertices of the accelerated Rauzy graph) and a natural number $n$ of steps of the ordinary induction in the corresponding step of the accelerated one.  Then the Rauzy induction provides the Markov shift $\Theta$ in this coding on a countable alphabet. One can associate in a natural way a graph $\Lambda$ with such a Markov shift. $\Lambda$ can be obtained from the Rauzy graph by dividing every vertex into a countable number of vertices and adding a required arrows between these new vertices. 

\begin{definition}
A countable Markov shift $\Theta$ with transition matrix $U$ and set of states $\EuScript{S}$ satisfies \emph{big images and pre-images property}(BIP) if there exist $\{i_{1},\cdot,i_{m}\} \in \EuScript{S}$ such that for all $j \in \EuScript{S}$ there are $1\leq k,l \leq m$ for which $u_{i_{k},j}u_{j,i_{l}}=1.$
\end{definition} 

\begin{definition}
A Markov shift is \emph{topologically mixing } if for any $i,j\in \EuScript{S}$ there exists a number $N=N(i,j)$ such that for any $n\geq N$ there is an admissible path of length $n$ on the graph of the shift that connects $i$ and $j$.
\end{definition}
\begin{lemma}
The Rauzy induction defines a countable topologically mixing Markov shift that satisfies BIP property.
\end{lemma}
\begin{proof}
The first part follows from the fact that both of the graphs of the induction are connected. 
In order to obtain BIP property we have to choose $m=6$ and $i_{j}$ each belong to a different vertex of the accelerated Rauzy graph. 
\end{proof}

\section{The suspension complex and the cocycle}\label{Co}
\subsection {Suspension complex}
Here we recall briefly the construction of the suspension complex for systems of isometries from \cite{GLP}. It can be considered as an analogue of the zippered rectangles model suggested by W. Veech (\cite{V}). 

With each special system of isometries we can associate a foliated 2-complex $\Sigma$ (in terms of $\mathbb R$-trees theory, it is a \emph{band complex}). Start with the disjoint union of the support interval (foliated by points) and strips $A_{j}\times [0, 1]$ (foliated by ${*} \times [0, 1]$). We get $\Sigma$ by glueing $A_{j}\times [0, 1]$ to $D$, identifying each $(t,0) \in A_{j}\times {0}$ with $(t,0) \in A_{j} \subset D$ and each $(t,1) \in A_{j}$  with $\phi_{j}(t) \in B_{j}\subset D$. We will identify $D$ with its image in $\Sigma$.
Thus, one gets a 2-dimensional complex with a vertical foliation on it.

Our family of band complexes is a particular class of what appears in geometric group theory as an instrument for describing actions of free groups on $\mathbb R$-trees (see, for example, \cite{BF} for details).
\begin{figure}[h]
\vspace{-1.7cm}
\includegraphics[width=10cm,height=12cm]{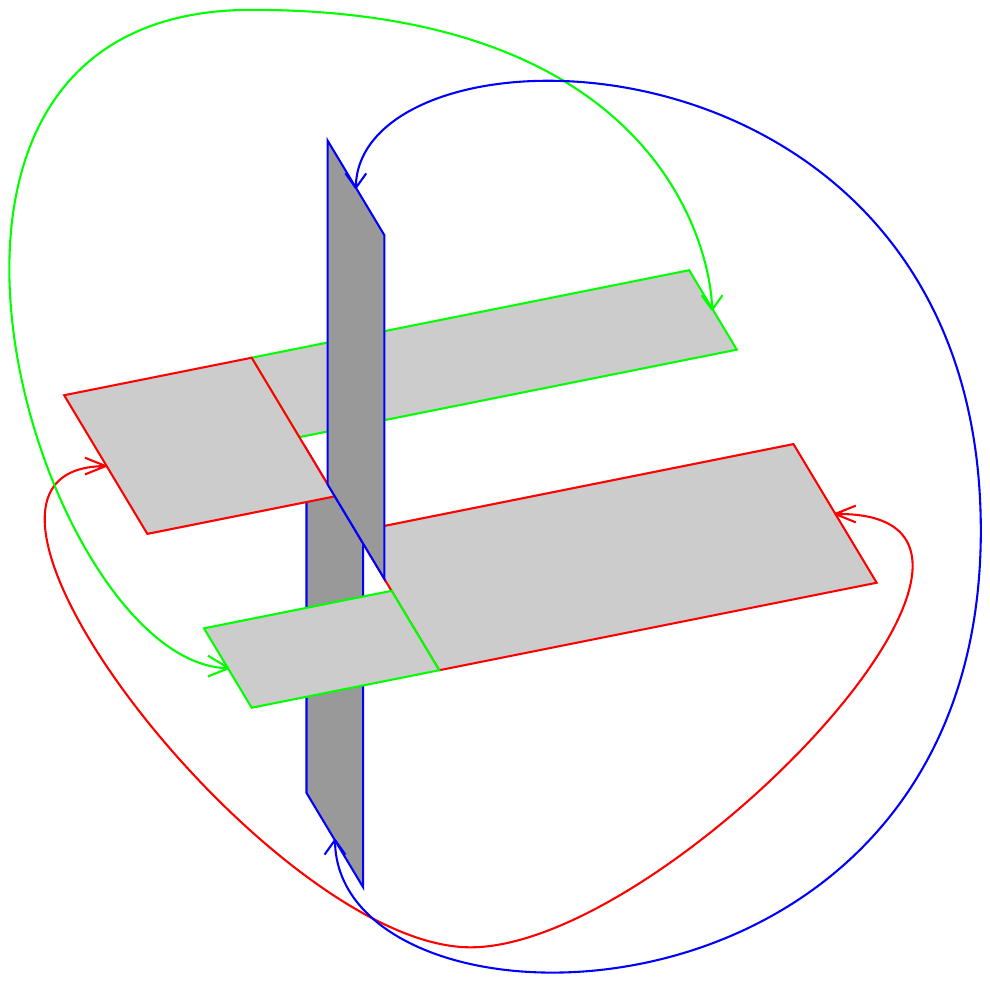}
\vspace{-3cm}
\caption{Band complex}
\label{C}
\end{figure}

\subsection{The cocycle}
One can apply the Rauzy induction not only to a system of isometries but also to a corresponding suspension complex. 
A suspension complex for a special system of isometries contains three bands, each of which has a width (horizontal lengths) and a length (more precisely, vertical length). The matrix $R$ described above tells us how are the widths of bands cut by the Rauzy induction. 
At the same time, the vertical lengths of the bands increase during the same procedure (see \cite{BF} for the description of the Rips machine application to the band complex). Indeed, once we make a transmission, the lengths of all the bands that are not involved in the operation as well as the length of the winner do not change; however, the length of the loser increases exactly by the length of the winner. The reduction does not influence the vertical lengths of bands. 

So, informally, the \emph{cocycle} is responsible for things that happen with vertical lengths of bands during the application of the Rauzy induction. 

More precisely, let $B$ be a matrix of the cocycle. For $n$ steps of the non-accelerated induction that do not form yet the step of an accelerated one (and therefore the combinatorics does not change) we denote by $B(n)$ is the following matrix:
$$\begin{pmatrix}
1 & 0 & 0\\
n & 1 & 0 \\
n & 0 & 1
\end{pmatrix}$$

Thus the matrix $B$ of the cocycle is a product $B(n)$ with different $n$ and the required permutation matrices. 
Let us denote by $B_{\gamma}$ the cocycle matrix that corresponds to the path $\gamma$ in $\Gamma$. 

We need also one more definition from \cite{AGY}:
\begin{definition}
A path $\gamma$ in the Rauzy graph is called \emph{positive}, if the matrix of $B_{\gamma}$ contains only strictly positive entries.
\end{definition}

\begin{lemma}
Any complete path in the Rauzy graph is positive.
\end{lemma}
\begin{proof}
We start from identity matrix of the cocycle.
Then, after one step of the induction, if $\alpha$ is a winner, the matrix of the cocycle changes in the following way:  the row with number $\alpha$ is added to the other two rows (here we always use the original enumeration and do not apply the permutations). 
So, if the path is complete, then each row was added to the other two at least once, and then all zero coefficients of the matrix are strictly positive numbers.
\end{proof}

\begin{remark} In \cite{AGY} and \cite{AR} the notion of \emph{strongly positive} path was also used. Namely, the path is strongly positive if it is positive and $(B^T_{\gamma})^{-1}(\theta_{\pi_s}\in\theta_{\pi_e})$ where $\pi_s$ is the starting permutation for $\gamma$, $\pi_e$ is the ending one and $\theta$ are Veech coordinates that were used to define the heights of the bands of zippered rectangles and the positions of saddles (up to which moments the rectangles are ``zippered'').

For us, each positive path is strongly positive in a sense of \cite{AGY} and \cite{AR} because in comparing with the case of IET the suspension construction for the systems of isometries has one significant difference: there is no difference between the past and the future for the orbits of the system, and so there is no orientation for the leaves of the vertical bands. Therefore the condition required for path to be strongly positive is satisfied in our case automatically. 
\end{remark}

\section{Properties of the Markov map}\label{UE}
In this section we prove that the Markov map $T$ that was introduced in \ref{mm} is uniformly expanding.
The main definition that we use came from \cite{AGY}:
\begin{definition}
Let $L$ be a finite or countable set, let $\Delta$ be a parameter space, and let $\{\Delta^{(l)}\}_{(l\in L)}$ be a partition into open sets of a full measure subset of $\Delta.$
A map $Q: \cup_{l} \Delta^{(l)} \rightarrow \Delta$ is a \emph{uniformly expanding} map if:
\begin{itemize}
\item For each $l$, $Q$ is a $C^{1}$ diffeomorphism between $\Delta^{(l)}$ and $\Delta$, and there exist constants $k>1$ (independent of $l$) and $C_{(l)}$ such that for all $x\in \Delta^{(l)}$ and all $v\in Q_{x}\Delta, k||v||\le||DQ(x)v||\le C_{(l)}||v||.$
\item Let $J(x)$ be the inverse of the Jacobian of $Q$ with respect to Lebesgue measure. Denote by $\EuScript{H}$ the set of inverse branches of $Q$. The function $log J$ is $C^{1}$ on each set $\Delta^{(l)}$ and there exists $C>0$ such that, for all $h\in \EuScript{H}$,
$$||D((logJ)\circ h)||_{C^{0}(\Delta)}\le C.$$ 
\end{itemize}
\end{definition}
\begin{remark}
In \cite{AGY} the parameter space is supposed to be John domain (see Definition 2.1 there). One can easily check that the set $X$ in our case is a John domain; however, we will no use any hyperbolic properties of the map and do not actually need this feature.
\end{remark}

So, we have the following: 
\begin{lemma}\label{UEL}
$T$ is a uniformly expanding map with respect to the Markov partition $(X^{i}_{n})$.
\end{lemma}
\begin{remark}
This property is already mentioned in \cite{AS} but we provide a direct proof.
\end{remark}
\begin{proof}
We consider the accelerated Rauzy induction with $n$ simple iterations included in it. Then, the induction map
is given by one of the two matrices $R(n)$ (see section \ref{Rn} for the formula). Without loss of generality, one can consider just one of them (say, the first one).

The corresponding projective map $T=T(n)=\frac{R(n)I}{||R(n)l||}$ is the following:
$$a' = \frac{b}{na-(n-1)},$$
$$b'=\frac{(n+1)a-n}{na-(n-1)}.$$

So, $DT(a,b) = DT(a) = \frac{1}{(na-(n-1))^3}.$

The Markov cell $X_{n}$ where $T$ is defined is determined by several inequalities that came from the matrix of the induction:
$$a< (n+1)b+(n+1)c;$$
$$a>nb+nc;$$
$$a-n(b+c)>c.$$
So, since $c=1-a-b$, the cell has the following vertices: 
$$\Big(\frac{n+1}{n+2}, \frac{1}{n+2}\Big),$$
$$\Big(\frac{n}{n+1}, \frac{1}{n+1}\Big),$$
$$\Big(\frac{2n+1}{2n+3}, \frac{1}{2n+3}\Big).$$
So, on the one hand, $a<\frac{2n+1}{2n+2}$ and $na-(n-1)<\frac{n+2}{2n+2}$. Therefore, $|DT|>\frac{8(n+1)^{3}}{(n+2)^{3}}>(4/3)^{3}.$

On the other hand, $a>\frac{n}{n+1},$ so $|DT|<(n+1)^3.$ It proves the first statement. 

The second condition is equivalent to the following one: 
$$\Big|\frac{DT(a_{1})}{DT(a_{2})}-1\Big|\le C\Big|T(a_{1}, b_{1})-T(a_{2}, b_{2})\Big|.$$
In our case: 
$\frac{DT(a_{1})}{DT(a_{2})} = \Big(\frac{1-x_{2}}{1-x_{1}}\Big)^{3},$ where $x_{i}=n(1-a_{i}), i=1,2.$
$$\Big|\Big(\frac{1-x_{2}}{1-x_{1}}\Big)^3-1\Big|=\Big|\Big(\frac{1-x_{2}}{1-x_{1}}\Big)-1\Big|\Big|\Big(\frac{1-x_{2}}{1-x_{1}}\Big)^2+\Big(\frac{1-x_{2}}{1-x_{1}}\Big)+1\Big|.$$
Due to the inequalities for $a_{i}$ that define the Markov cell we work with, one can conclude that $\frac{n}{n+2}<x_{i}<\frac{n}{n+1}, i=1,2.$
It means that $X=\frac{1-x_{2}}{1-x_{1}}+1$ satisfies the following inequalities: $1<X\le 4.$
So, $$\Big|\Big(\frac{1-x_{2}}{1-x_{1}}\Big)^2+\Big(\frac{1-x_{2}}{1-x_{1}}\Big)+1\Big|=\Big|X^{2}-(X-1)\Big|\le X^{2}+X+1\le 16+1+1=18.$$
Now we consider $$\Big|T(a_{1},b_{1})-T(a_{2},b_{2})\Big|\ge \Big|\frac{(n+1)a_{1}-n}{na_{1}-(n-1)}-\frac{(n+1)a_{2}-n}{na_{2}-(n-1)}\Big|=\frac{\Big|\frac{x_{1}}{1-x{1}}-\frac{x_{2}}{1-x{2}}\Big|}{n}=$$
$$=\frac{\Big|x_{1}-x_{2}\Big|}{n(1-x_{1})(1-x_{2})}\ge \frac{1}{2}\frac {\Big|x_{1}-x_{2}\Big|}{1-x_{1}}.$$
So, we proved the second condition with $C=18\cdot 2=36.$
\end{proof}
\begin{remark}
Using the technique from \cite{MN} one can check that the uniform expanding property of $T$ implies that the Rauzy gasket has zero Lebesgue measure.
\end{remark}

\section{Distortion estimates}\label{DE}
In this section we prove that the Markov map $T$ has a bounded distortion in the sense of \cite{AGY}.
\subsection{Conditional probabilities}
The distortion argument will involve not only the study of Lebesgue measure, but also of its forward images under the renormalization map. So, following the strategy from \cite{AGY} and \cite{AR}, we first construct a class of measures which is invariant as a whole.

Let us consider the adjusted Rauzy graph and some path $\gamma$ in it. It was proved above that the graph is connected and therefore in terms of combinatorics of IET we have only one Rauzy class.

As before, $B_{\gamma}$ is the matrix of the cocycle corresponding to $\gamma$. The original parameter space $X=\Delta$ is a simplex in $\mathbb RP^{2}$. We will be interested in the measure of the following part of it:
$$\Delta'_{\gamma}= B^{T}_{\gamma}\mathbb R^{3}_{+},$$
where $B^{T}$ is a transposed matrix with respect to $B$.

For $q=(q_1, q_2, q_3) \in \mathbb R^{3}_{+}$ we define a measure $\nu_{q}$ on the $\sigma$-algebra $A\subset\mathbb R^{3}$ of Borel sets which are positively invariant ($\mathbb R_{+}A=A$):
$$\nu_{q}(A)=3!\cdot Leb(A \cap \{\lambda\in\mathbb R^{3}_{+}: \left\langle {\lambda,q} \right\rangle<1\}.$$

Equivalently, $\nu_{q}$ can be considered as a measure on the projective space $RP^{2}_{+}.$
One can also check that $\nu_{q}(B^{T}_{\gamma}A)=\nu_{B_{\gamma}q}(A)$ and $\nu_{q}(\mathbb R^{3}_{+})=\frac{1}{q_1q_2q_3}.$

The measures $\nu_q$ are used to calculate the probabilities of realization of different types of combinatorics related to the induction.
Let $\pi$ be the permutation from which $\gamma$ starts and let us denote by $\alpha\in \left\{ {1,2,3} \right\}$ the winner and by $\beta\in \left\{ {1,2,3} \right\}$ the loser of the first iteration of the Rauzy induction (without acceleration). 
Then, the conditional probability related to the given combinatorics can be defined in the following way: $P_{q}(\gamma|\pi)=\frac{q_{\beta}}{(q_{\alpha}+q_{\beta})}.$

More generally, for $\EuScript{A'}\subset \EuScript{A}=\left\{ {1,2,3} \right\}$ and $q\in \mathbb R^{\EuScript{A}}_{+},$ let $N_{\EuScript{A'}}(q)=\prod_{\alpha\in \EuScript{A'}}q_{\alpha}.$ Let also $N(q)=N_{\EuScript{A}}(q).$ Then $$P_{q}(\gamma|\pi)=\frac {N(q)}{N(B_{\gamma}q)}.$$

\subsection {Kerckhoff lemma}
In this section we prove our key estimate that gives a base for the further more subtle estimations of distortion properties of the cocycle matrix. 
The idea that was used for the first time in \cite{K} for IET is the following: in order to control how does the induction distort the vector that originally was balanced, one has to check that the ratio between the norms of the rows (equivalently, columns) of the matrix of the cocycle (equivalently, of the induction matrix) can rarely be very high.
More formally, we have the following: 
\begin{lemma}
For any $T>0, q \in \mathbb{R}^{\EuScript{A}}_{+}, \alpha \in \EuScript{A}$
\begin{equation}\label{Ke}
P_{q}(\Gamma_{\alpha}(\pi), (B_{\gamma}q)_{\alpha}>Tq_{\alpha}|\pi)<T^{-1},
\end{equation}
where $\Gamma_{\alpha}(\pi)$ denotes the set of paths starting at $\pi$ with no winner equal to $\alpha$.
\end{lemma}

\begin{proof}
First, let us note that the Rauzy induction works in the following way: after one step two rows of the matrix of the cocycle increase their norm while the third one remains stable. Once we make a step of the accelerated induction the combinatorics changes and then we start to add another row (not that one that we added before) to two others. One can easily check that in this case it is enough to make one accelerated step and then one usual step of the induction to get the balanced matrix again (by the last we mean the matrix such that the ratio of norms of rows can be evaluated by some constants that do not depend on the coefficients of the matrices). It easily implies the estimation (\ref{Ke}). 

Therefore, our main goal is to prove (\ref{Ke}) for several consequent steps of non-accelerated induction what do not comprise the whole step of the accelerated one.

In this case the path  $\gamma$ we work with is just a loop based at one vertex, say, {1,2,3}. Lets say that we make $n$ ordinary iterations of the Rauzy induction. Sometimes we will refer to this parameter using the term ``time". We also choose some $q=(q_1,q_2,q_3)\in \mathbb R^{3}.$
Then, $$Bq = \begin{pmatrix}
q_{1} \\
nq_{1}+q_{2}  \\
nq_{1}+q_{3}
\end{pmatrix}$$
It is easy to see that the modulus of the first component remains stable while two other moduli increase (as a norm we mean here the value of maximal element).

Now, we follow the strategy suggested in \cite{K} (see also \cite{AGY}, Appendix A). We consider the matrix of the Rauzy induction $R_{\gamma}(k)=B_{\gamma}^T$, where $k$ is the time. 
Let us denote the columns of the matrix $R_{\gamma}(k)$ by $v_{1}(k), v_{2}(k), v_{3}(k)$, respectively. Then the norm of $v_{2}(k)$ and $v_{3}(k)$ increase with $k$ while norm of $v_{1}(k)$ remains stable. 
Indeed, the image of the original parameter space is the following:
$$V(k) = \{x\in\mathbb R^3_{+}: x=\sum_{i=1}^{3} v_{i}(k) \alpha_{i}, \sum_{i=1}^3 \alpha_{i}=1\}.$$

One can check that 
$$\tilde v_{2}(k)= kv_{1}+v_{2}, \tilde v_{3} = kv_{1}+v_{3},$$
where $\tilde v$ means the value of $v$ after $k$ iterations and $k\le \frac {\alpha_{1}}{\alpha_{i}}, i=2,3,$ since $\gamma$ is a loop and during the time $n$ the path never leaves the vertex $(1,2,3)$.

We evaluate the probability of the following event: $$(B_{\gamma}q)_{\alpha}>Tq_{\alpha}$$ for every $\alpha$.
So, we have the following: 
$$|v_{2}+nv_{1}|>T|v_{2}|,$$
and then
$$\frac {\alpha_{1}|v_{1}|}{\alpha_{2}|v_{2}|}\ge\frac {n|v_{1}|}{|v_{2}|}>T-1.$$
Here we used the fact that all components of $v$ are positive. 
The same inequality holds for $v_{3}$ because at the beginning we had a balanced vector.

As it was mentioned above, the probability we are interested in can be expressed as the ratio of the measures of the part of the parameter space that is defined by the given inequalities and combinatorics.
Lets us estimate this ratio: 
$$\frac {3!|v_{1}||v_{2}||v_{3}|}{3!|v_{1}||v_{2}+nv_{1}||v_{3}+nv_{1}|} = \frac {|v_{2}|}{|v_{2}+nv_{1}|}\cdot\frac {|v_{3}|}{|v_{3}+nv_{1}|}<\frac {|v_{2}|}{n|v_{1}|}\cdot\frac {|v_{3}|}{n|v_{1}|}<
\frac {1}{(T-1)^2}<\frac{1}{T}.$$
\end{proof}

\subsection {Further Distortion Estimates}
In this section we apply Kerckhoff lemma to obtain some more subtle estimations on the distortion.
We mainly follow the strategy suggested in Appendix A of \cite{AGY}.

Before we actually state the theorem, let us introduce some useful notations:
$$\EuScript{A'} \subset \EuScript{A};$$
$$m_{\EuScript{A'}}(q) = min_{\alpha\in\EuScript{A'}}q_{\alpha};$$
$$m(q)=m_{\EuScript{A}}(q)$$
$$m_{k}(q) = max_{\{\EuScript{A'} \subset \EuScript{A}: |\EuScript{A'}|=k\}}m_{\EuScript{A'}}(q);$$
$$M_{q} = max_{\alpha\in\EuScript{A}}q_{\alpha}.$$

The principal result of this part is the following 
\begin{theorem}\label{A2}
There exists $C>1$ such that for all $q\in \mathbb{R}^{\EuScript{A}}_{+}$
$$P_{q}(M(B_{\gamma}q)<Cmin\{m(B_{\gamma}q),M(q)\})>C^{-1}.$$
\end{theorem}

\begin{proof}
The main idea of the proof comes from \cite{AGY}: one should consider all the subsets $\EuScript{A'} \subset \EuScript{A}$ of fixed cardinality $k$ and prove that for $1\le k\le 3$ there exists $C>1$ such that
\begin{equation}\label{with_k}
P_{q}(M(B_{\gamma}q)<Cmin\{m_k(B_{\gamma}q),M(q)\})>C^{-1}.
\end{equation}

In our case $\EuScript{A}=1,2,3$, so we have to consider only two types of subsets $\EuScript{A'}$ of $\EuScript{A}$: with cardinality 1 or 2. The proof is by induction on $k$.  We will provide the whole procedure for the step from $k=1$ to $k=2$; the next step is similar. 
So, our main goal is to prove the following statement for $k=2$:
$$P_{q}(M(B_{\gamma}q)<Cmin\{m_{k}(B_{\gamma}q,M(q))\}|\pi)>C^{-1}.$$

\begin{itemize}
\item Base of the induction: for $k=1$ the statement is obvious.

\item Step of the induction: let us fix $k=1$ and prove the statement for the subsets of two elements. 
\end{itemize} 

Let $\Gamma$ be the set of minimal paths starting at the fixed permutation $\pi$ such that there exists $C_{0}>1:$ for any $\gamma \in \Gamma$
$M(B_{\gamma}q)<C_{0}min\{m_{1}(B_{\gamma}q),M(q)\}$. 
This property implies that $M(B_{\gamma}q)<C_{0}M(q).$

\noindent  Now, $m_1(B_{\gamma}q)=max_{\EuScript{A'}: |\EuScript{A'}|=1}m_{\EuScript{A'}}(B_{\gamma}q).$
So, by definition, there exists a subset $\Gamma_{1}\subset \Gamma$ and $\EuScript{A'}$ with cardinality such that $P_{q}(\Gamma_{1})>C_{1}^{-1}$ and for any $\gamma\in \Gamma_{1}$ 
$$m_{1}(B_{\gamma}q) = m_{\EuScript{A'}}(B_{\gamma}q).$$
Without loss of generality one can assume that $\EuScript{A'}=\left\{1\right\}$.

\noindent Now, we use the acceleration - we iterate (and then continue the path from $\Gamma_{1}$) up to the moment the reduction will act on the second pair of subintervals. 
More precisely, we fix $\gamma_{1}\in \Gamma_{1}$ and consider $\gamma_{2} = \gamma_{1}\hat{\gamma_{1}}$ with minimal length such that the path $\hat \gamma_{1}$ ends by the permutation that starts with $2$, not with $1$. 
These paths $\gamma_{2}$ form a set $\Gamma_{2}$. So, for some $C_{2}>1,$
\begin{equation}\label{c2}
P_{q}(\Gamma_{2})>C_{2}^{-1}
\end{equation}
and 
\begin{equation}\label{C2}
M(B_{\gamma_{2}}q)<C_{2}M(B_{\gamma_{1}}q).
\end{equation}
The statement \ref{C2} follows from the fact that the accelerated induction applied to the balanced vector results in a balanced vector again (see the proof of the Kerckhoff lemma). Indeed, $$M(B_{\gamma_{2}}q)=max \{nq_1+1; (n+1)q_1+q_2;2nq_1+q_2+q_3\}$$ while $$M(B_{\gamma_{1}}q)=max\{q_1,nq_1+q_2, nq_1+q_3\}$$ since $\gamma_2$ is minimal.

\noindent Now we construct $\Gamma_{3}$ that contains $\gamma_{3} = \gamma_{2}\hat{\gamma_{2}}$ where $\gamma_{2}\in \Gamma_{2}$ and $\hat\gamma_{2}$ is such that all arrows contain $b$ or $c$ as the winners except the last one (so, at the end we obtain a permutation starting with 1 again) and 
\begin{equation}\label{6}
(B_{\gamma_{3}}q)_{1}\le 6(B_{\gamma_{2}}q)_{1}.
\end{equation}

\noindent Since the condition \ref{6} does not always hold, one has to estimate the measure of the parameter space where (\ref{6}) is true. Let us express (\ref{6}) for this purpose in terms of the Rauzy induction. We denote by $x$ the number of steps when $b$ was the winner, and by $y$ - the number of steps where $c$ was the winner. So, the induction is described by the following matrix: 
$$\begin{pmatrix}
n(xy+2y+x+1)+1 & xy+x+y+1 & y\\
n(2xy+4y+2+x)+1 & 2xy+x+2y+2 & 2y \\
n(xy+2y+2x+3)+1 & (x+1)y+2(x+1) & y+1
\end{pmatrix}$$
We evaluate now the norms of the rows of the matrix. The condition \ref{6} means that 
\begin{equation}\label{xy}
xy+2y+x+1\le 6
\end{equation}
We will use the estimation \ref{xy} in some calculations later. 

The Kerckhoff's lemma together with the property (\ref{6}) imply that $P_{q}(\Gamma_{3}|\gamma_{2})>\frac{1}{6}$ and $P_{q}(\Gamma_{3})=P_{q}(\Gamma_{3}|\gamma_{2})P_{q}(\Gamma_{2})>\frac{1}{6C_{2}}$.

Now, if $M(B_{\gamma_{3}}q)>6M(B_{\gamma_{2}}q)$, then we take the minimal path $\gamma_{4}$ between $\gamma_{3}$ and $\gamma_{2}$ such that the same condition holds. For the obtained $\gamma_{4}$ one can check directly using the previous calculations that for $\alpha$ equals to 2 or 3 $$M(B_{\gamma_{4}}q)=(B_{\gamma_{4}}q)_{\alpha}\le 12M(B_{\gamma_{2}}q)$$ (more precisely, it follows from (\ref{xy})).

\noindent Moreover, in this case we have that $$(B_{\gamma_{4}}q)_{1}>(12C_{0}C_{2})^{-1}M(B_{\gamma_{4}}q).$$  Note also, that $P_{q}(\Gamma_{4})\ge P_{q}(\Gamma_{3})>\frac{1}{6C_{2}},$ where $\Gamma_4$ is a set of all $\gamma_4$ we described above.
So, we have $\EuScript{A'}=\{1,\alpha\}$ for which the statement of the theorem holds with the constant $C=12C_{0}C_{2}$.

If $M(B_{\gamma_{3}}q)\le 6M(B_{\gamma_{2}}q)$, then for this $\gamma_{3}$ we can consider the following $\EuScript {A'}$ for which the statement of the theorem holds:
$\EuScript{A'}=\{1,\beta\},$ where $\beta$ is the loser of the last arrow (1 was the winner, so $\beta$ is 2 or 3).

Anyway, there exists the set $\Gamma_{4}$ such that $$P_q(\Gamma_{4} | \pi) \ge P_q(\Gamma_3 | \pi) > (2C_2)^{-1},$$
therefore the statement of the theorem holds for the set of cardinality 2.

\end{proof}
\begin{remark} Note that all inequalities and estimations for the step $2\rightarrow 3$ are the same; the only difference is a description for $\Gamma_{3}$. In this case $\hat{\gamma_{2}}$ corresponds to the following scheme of the Rauzy induction: several iterations with the same winner (for instance, 1) and then one iteration with the different winner (2 or 3). Then, the matrix of the cocycle depends on 2 parameters instead of 3. 
\end{remark}

\section{Exponential tails}\label{ET}
In the current section we obtain more subtle estimations on the conditional probabilities that were discussed above. Informally, our main goal is to replace $C^{-1}$ in the right part of our statement by $C^{-\delta},$ where $\delta$ is some positive constant. 
We have the following

\begin{theorem} \label{thm:proba-estimate}
For every $\hat{\gamma}$ there exist $\delta>0, C>0$ such that for every $q\in \mathbb R^{\EuScript{A}}_{+}$ and every $T>1$

$P_{q}(\gamma$ cannot be written as $\gamma_{s}\hat{\gamma}\gamma_{e}$ and $M(B_{\gamma}q)>TM(q))\le CT^{-\delta}.$
\end{theorem}

\begin{remark}
The restriction on the paths means that we only consider paths that do not contain $\hat{\gamma}$ as a proper part.
\end{remark}

The most important point of the argument we use is that the estimates that we prove in Theorem \ref{thm:proba-estimate} are uniform with respect to $q$.

The proof of the theorem is based on the following two lemmas that can be considered as the corollaries from Theorem \ref{A2}.

\begin{lemma}
Let us denote by $\Pi$ a group of paths that start on a permutation $\pi$. 
There exists $C>1$ such that for any permutation $\pi$
$$P_{q}(\gamma \in \Pi, M(B_{\gamma}q)>CM(q), m(B_{\gamma}q)<M(q)|\pi)<1-\frac{1}{C}.$$
\end{lemma}
\begin{proof}
We have to evaluate the probability of the event complimentary to the event we worked with in Theorem \ref{A2}:
$$M(B_{\gamma}q)<Cmin\{m(B_{\gamma}q, M(q)\}\leftrightarrow X \cup Y,$$
where $X$ is identified by
$$\left\{
\begin{aligned} 
M(B_{\gamma}q)<Cm(B_{\gamma}q)\\ 
m(B_{\gamma}q) < M(q).\\ 
\end{aligned} \right.
$$
and $Y$ is the following:
$$\left\{
\begin{aligned} 
M(B_{\gamma}q) < CM(q)\\ 
M(q)<m(B_{\gamma}q). \\ 
\end{aligned} \right.
$$
Complement to $X\cup Y = X'\cap Y',$ where $X'$ is the following:
$$\left\{
\begin{aligned} 
M(B_{\gamma}q)>CM(q)\\ 
m(B_{\gamma}q) > M(q).\\ 
\end{aligned} \right.
$$
and $Y'$ is 
$$\left\{
\begin{aligned} 
M(B_{\gamma}q)>Cm(B_{\gamma}q)\\ 
m(B_{\gamma}q) < M(q).\\ 
\end{aligned} \right.
$$
So, the probability of the event
$$\left\{
\begin{aligned} 
M(B_{\gamma}q)>CM(q)\\ 
m(B_{\gamma}q) < M(q).\\ 
\end{aligned} \right.
$$
is less than the probability of $Y'$ and so less than $1-\frac{1}{C}.$
\end{proof}
The next lemma follows from the previous one and is also important for the proof of Theorem \ref{thm:proba-estimate}.
\begin{lemma}\label{PL}
For any $\hat{\gamma}$ there exist $M\ge 0, \rho<1$ such that for any $\pi, q\in \mathbb{R}^{\EuScript{A}}_{+}$ 

$P_{q}(\gamma$ can not be written as $\gamma_{s}\hat{\gamma}\gamma_{e}$ and $M(B_{\gamma}q)>2^{M}M(q)|\pi) \le \rho.$
\end{lemma}

\begin{proof}
Lets fix $M_{0}$ large enough and let $M=2M_{0}$. Let us consider the set of minimal paths
$$\Gamma = \{\gamma: M(B_{\gamma}q)>2^{M}M(q)\}$$
such that $\gamma$ can not be written in a way $\gamma_{s}\hat{\gamma}\gamma_{e}$.

Then any path in $\Gamma$ can be written as $\gamma=\gamma_{1}\gamma_{2}$ where $\gamma_{1}$ is a  minimal path such that 
$$M(B_{\gamma_{1}}q)>2^{M_{0}}M(q).$$
Let us denote the set of such $\gamma_{1}$ by $\Gamma_{1}$.
So $\Gamma_{1}$ is disjoint in terms of \cite{AGY} which means that any path is not a part of some other path from the same set (it follows directly from minimality). 
Now we consider the subset $\tilde \Gamma_{1}$ of this set $\Gamma_{1}$ consisting of all $\gamma_{1}$ such that 
$$M_{\EuScript{A'}}(B_{\gamma_{1}}q)\ge M(q),$$ where $\EuScript{A'}$ is a proper set of $\EuScript{A}$ (the last property means that $m(B_{\gamma_{1}}q)\ge M(q)$).

By the previous lemma we have that 
$$P_{q}(\Gamma_{1}\setminus \tilde\Gamma_{1}|\pi)<1-\frac{1}{C}$$
with some constant $C>1.$

Now we use the strategy from \cite{AR}: we fix some permutation $\pi_{e}$ and consider the path $\gamma_{\pi_{e}}$ that will be the shortest path starting at $\pi_{e}$ and containing $\hat{\gamma}$ as a second part: $\gamma_{\pi_{e}}=\gamma_{s}\hat{\gamma}.$
Then, if $M_{0}$ is large enough, we can assume that 
\begin{equation}
||B_{\gamma_{\pi_{e}}}||<\frac {2^{M_{0}-1}}{3}
\label{B}
\end{equation}
If $\pi_{e}$(that could be any permutation) is the end of $\gamma_{1}\in\Gamma_{1}$, then
$$P_{q}(\Gamma|\gamma_{1})\le 1 - P_{B_{\gamma_{1}}q}(\gamma_{\pi_{e}}|\pi_{e})$$
because $\gamma$ does not contain $\hat{\gamma}$ as a part (it is a condition of the lemma). 

If $\gamma_{1}\in \tilde\Gamma_{1}$, the last probability can be estimated directly in terms of the measures of subsimplices of the original simplex: if $N(q) = q_{1}q_{2}q_{3},$ then
$$P_{B_{\gamma_{1}}q}(\gamma_{\pi_{e}}|\pi_{e}) = \frac {N(B_{\gamma_{1}q})}{N(B_{\gamma_{\pi_{e}}B_{\gamma_{1}}q)}}\ge 2^{-6M_{0}},$$
because $N(B_{\gamma_{1}q})\ge M(q)^{3}$ if $\gamma_{1}\in \Gamma_{1}$ and, on the other hand, $M(B_{\gamma_{1}})q<2^{M}M(q)$ that, together with (\ref{B}), provides the estimation for the denominator.

Now, if $P_{q}(\tilde\Gamma_{1})\ge \frac{1}{2C},$ then
$$
P_{q}(\Gamma_{1})=P_{q}(\Gamma_{1}\setminus \tilde\Gamma_{1})+P_{q}(\tilde \Gamma_{1})P_{q}(\Gamma_{1}|\tilde\Gamma_{1})\le$$

$$\le 1-P_{q}(\tilde\Gamma_{1})+P_{q}(\tilde\Gamma_{1})\cdot(1-2^{-6M_{0}})=1-P_{q}(\tilde\Gamma_{1})\cdot2^{-6M_{0}}\le 1-\frac{2^{-6M_{0}}}{2C}.
$$
If $P_{q}(\tilde\Gamma_{1})<\frac{1}{2C},$
then $P_{q}(\Gamma_{1})<1-\frac{1}{C}+\frac{1}{2C}=1-\frac{1}{2C}.$

So, lemma \label{PL} holds with $\rho = 1 - \frac{2^{-6M_{0}}}{2C}$.
\end{proof}

\begin{proof}
Now we turn to the proof of the theorem.
Let us use $M$ and $\rho$ from the previous lemma. We denote by $k$ the smallest integer such that $T\ge 2^{k(M+1)}$. We denote by  $\gamma$ the minimal path that does not include $\hat{\gamma}$ as a part and such that $M(B_{\gamma}q)>2^{k(M+1)}M(q)$. Then $\gamma$ can be considered as a composition of $\gamma_{1}\gamma_{2}...\gamma_{k}$ where $M(B_{\gamma_{i}}q)>2^{i(M+1)}M(q)$ and $\gamma_{i}$ is minimal with this property. The sets of corresponding $\Gamma_{i}$ are disjoint due to minimality. Lemma \ref{PL} implies that
$$P_{q}(\Gamma_{i+1}|\gamma_{i})\le \rho.$$
So $P_{q}(\Gamma)<{\rho}^{k}$. The result follows from the definition of $k$.
\end{proof}

\section{The roof function and the suspension flow}\label{RF}
\subsection{The roof function}
The construction of the roof function that we present in the current section is based on the idea of renormalization provided by Veech in \cite{V}.

Informally, we take a point $x\in X$ of the parameter space and consider the special system of isometries $S=S(x)$ that corresponds to this point. We denote by $\lambda=(a,b,c)$ the vector of the lengths of subintervals of $S$, and $\pi$ is the corresponding permutation. Then one applies the Rauzy induction to the system $S$; it results in the system $S'$ with the vector of lengths of subintervals $\lambda'=(a',b',c').$ The roof function $r(x)=r(\lambda,\pi)=-\log||\lambda'||$. In other words, the roof function is the first return time to some small subsimplex in the parameter space. We proceed with the formal definition (the same one was used in \cite{AGY} and \cite{AR}).

Let us fix some positive complete path $\gamma_{*}$ starting and ending at the same permutation $\pi$, and the subsimplex of the parameter space that corresponds to this path $\Delta_{\gamma_{*}}$. 

We are interested in the first return map to the subsimplex $\Delta_{\gamma_{*}}$. So, the connected components of the domain of this map are given by the $\Delta_{\gamma\gamma_{*}}$ where $\gamma$ is a path that contains $\gamma_{*}$ as part but does not start with $\gamma_{*}\gamma_{*}.$ 

Then, the first return map $T$ restricted to such a component is the following: 
\begin{equation}\label{T}
T(\lambda, \pi)=\Big(\frac{R_{\gamma}\lambda}{||R_{\gamma}\lambda||},\pi),
\end{equation}
where $R=(B^{T})^{-1}.$

\begin{definition}
The roof function is the return time to the connected component described above:
$$r(\lambda, \pi) = -\log||(B^{T}_{\gamma})^{-1} \lambda||,$$
where $\lambda = (a,b,c)$ is a vector of lengths and $\pi$ is a corresponding permutation.
\end{definition}

\begin{remark}
As it was mentioned in \cite{AGY} and \cite{AR}, with such a definition one works with the precompact sections because the path $\gamma_{*}$ is positive.
\end{remark}

\subsection{The suspension flow}
We use a standard definition of the suspension flow constructed by the shift transformation $\Theta$ and the roof function $r$ (it is an analogue of the construction from \cite{V} for the Teichm\"uller flow). The suspension flow renormalizes the length of the interval to $1$. 

Formally, the definition is as follows: 
The flow $\Phi(\Theta, r)$ is defined on a space $Y=((\lambda, \pi,t)\in \Theta \times \mathbb R: 0 \leq t \leq r(\lambda, \pi))$ 
and the points $(\lambda, \tau, r(\lambda, \pi))$ and $(\Theta(\lambda,\pi),0)$ are identified. It acts in the following way: 
$$\Phi_t(\lambda,\pi,s)=(\lambda,\pi,t+s)$$ 
whenever $s+t \in [0,r(\lambda,\pi)].$

The space $Y$ where the flow we work with is defined is the direct product of the parameter space $X$ and $\mathbb R$ modulo equivalence relation. 

\subsection{Correctness of the suspension model}

The flow we work with is the suspension over $T$ with roof function $r$. In this suspension model, the orbits that do not come back to the fixed precompact section escape the control. 
However, the following properties of the Markov map hold:
\begin{enumerate}
\item the BIP property implies that each small simplex of the Markov partition (let's say that it $\Delta_{\gamma}$ where $\gamma$ is a corresponding complete path in the Rauzy graph) is mapped on the whole parameter space $X$, and the map is surjective;
\item the Markov map $T$ is uniformly expanding and so the Jacobian of the map from $\Delta_\gamma$ to $X$ is bounded.
\end{enumerate}

Let us denote by $RG(\Delta_{\gamma})$ the set of points of the Rauzy gasket inside of $\Delta_{\gamma}$. The properties mentioned above imply that $Hdim(RG(\Delta_{\gamma}))=Hdim (RG),$ where $RG$ is the Rauzy gasket and Hdim is the Hausdorff dimension. 

The same argument can be used for $\Delta_{\gamma}$ and $\Delta_{\gamma'},$ where $\gamma'=\gamma\hat\gamma\gamma$ for some suitable $\hat\gamma$. Therefore, the orbits that escape the control do not contribute to the Hausdorff dimension of the Rauzy gasket, and our suspension model is correct. 

\section{Exponential tails of roof function}\label{ETRF}
In this section we prove that the roof function constructed above has \emph{exponential tails}. We follow the strategy from \cite{AGY} and \cite{B}.
\begin{definition}
A function $f$ has \emph{exponential tails} if there exists $\sigma > 0$ such that $\int_{\Delta}e^{\sigma f}dLeb<\infty.$
\end{definition}
\begin{theorem}
The roof function $r$ defined above has exponential tails.
\end{theorem}
\begin{proof}
This theorem is a direct corollary from Theorem \ref{thm:proba-estimate}.
The main idea is the same as was used in the case of IET: $-\log||(B^{T}_{\gamma_{*}})^{-1} \lambda||$ is the ``Teichm\"uller" time needed to renormalize our interval to length 1. Then  time is divided into pieces of exponential size. For each piece, we apply Theorem \ref{thm:proba-estimate}. 

Indeed, in section \ref{DE} we constructed the set of measures $\nu_{q}$ that depended on vector $q$. Let us consider $q_{0}=(1,1,1)$ and the corresponding measure $\nu_{q_{0}}$. The pushforward of this measure under radial projection yields a smooth function on the parameter space of the flow, say, $\nu.$ In order to prove the theorem, it is enough to show that 
\begin{equation}\label{nu}
\nu\{x\in \Delta_{\gamma_{*}}:r(x)\ge log T\}\le CT^{-\delta}
\end{equation}
for some $C$ and some $\delta.$

The connected component of the domain of the Markov map $T(\lambda,\pi)$ that intersects the set $W=\{x:\{x\in \Delta_{\gamma_{*}}:r(x)\ge log T\}\le CT^{-\delta}\}$ is of the form  $\Delta_{\gamma}\cap U$ for some $\gamma$.

One can easily see that $\gamma$ can not be a concatenation of more than three copies of $\gamma_{*}$. This restriction comes from the fact that we work with the first return maps. Also, the definition of the roof function and the definition of the set $W$ imply that  $$M(B_{\gamma}q_{0})>C^{-1}T,$$ where $C$ is some constant depending on $\gamma_{*}$. 

Now we estimate the measure of the interesting set in terms of probabilities of corresponding events:
$\nu\{x\in \delta_{\gamma_{*}}:r(x)\ge log T\}\le P_{q_{0}}(\gamma$ does not contain some $\hat\gamma$ as a proper set and $M(B_{\gamma}q_{0})>C^{-1}T|\pi)<CT^{-\delta}.$
The statement of the theorem follows now from Theorem \ref{thm:proba-estimate}.
\end{proof}

\section{Proof of the main theorem}\label{HD}
In this section we prove Theorem \ref{thm:main} mainly using the ideas from \cite{AD}. First, we introduce the notion of the fast decaying Markov map and prove that the Markov map we work with satisfies this property. Then, deduce our main result from Theorem 26 in \cite{AD}. 

\subsection{Fast decaying Markov maps}

Let $\Delta$ be a measurable space and $T: \Delta \rightarrow \Delta$ be a Markov map. We will denote the corresponding Markov partition by $\Delta_{(l)}, l\in \mathbb {Z}$. 
Let us consider $\underline l=(l_{1}, l_{2},\dotsc,l_{n})$ where all $l_{i}$ are integers; by $\Delta^{\underline l}$ we denote $x\in\Delta$ such that $T^{j-1}(x)\in \Delta_{l_{j}}$ for all $1\le j\le n.$
We say that $n$ is a \emph{depth} of the partition done by $\Delta^{\underline l}$.
\begin{definition}
We say that $T$ is \emph{fast decaying} if there exists $C_{1}>0, \alpha_{1}>0$ such that 
\begin{equation}\label{fade}
\sum_{\mu(\Delta^{\underline l})\le\varepsilon}\mu(\Delta^{\underline l})\le C_{1}\varepsilon^{\alpha_{1}}
\end{equation}
for $0<\varepsilon<1.$
\end{definition}

\begin{lemma}
Exponential tails of the roof function implies fast decaying property of the Markov map.
\end{lemma}
\begin{proof}
First, one can check that $|DT(\lambda,\pi)|=e^{3r(\lambda,\pi)},$ where $\lambda=(a,b,c).$
Indeed, let us fix some point in the parameter space (equivalently, the system of isometries) and consider the action of the Rauzy induction on this system. 
Several steps of the [non-accelerated] induction are described by the following matrix (up to the order of rows): 
$$A=\begin{pmatrix}
1 & -n & -n\\
0 & 1 & 0 \\
0 & 0 & 1
\end{pmatrix}.$$
Then, $r(\lambda)=-log||A\lambda||=-log(a-n(b+c))=-log((n+1)a-n).$ 
On the other hand, in the proof of Lemma \ref{UEL} we showed that $|DT(\lambda,\pi)|=\frac{1}{{((n+1)a-n)}^{3}}=e^{3r(\lambda)}.$
Now, the lemma follows from this statement and the fact that the measure of subsimplices (Markov cylinders) is proportional to $|DT|$.
Indeed, the scheme of the check is as follows: one can replace the measures of subsimplices we sum up in \ref{fade} by the corresponding jacobians; the jacobians can be replaced by the exponent of the roof function; the last sum can be evaluated using the exponential tails of the roof function (namely, the convergence of the corresponding integral).
\end{proof}

\subsection{The proof of Theorem \ref {thm:main}}

Avila and Delecroix in \cite{AD} proved the following 
\begin{theorem}[AD,2013]\label{AD}
Assume that $T$ is fast decaying. For $n\ge 1$, let $X_{n}\subset \Delta$ be a union of some subsimplices ($\Delta^{l})$ of depth $n$ and let $X=\liminf X_{n}.$ Let 
$$\delta =-\lim_{n\rightarrow\infty} \frac{1}{n} \ln\mu(X_{n}),$$ then $HD(X)\le p-1-min(\delta,\alpha_{1}),$ where $\alpha_{1}$ is the fast decay constant.
\end{theorem}

Let us consider the Markov map $T$ and the corresponding Markov partition. Let us fix the total amount of the steps of the Rauzy induction (say, $n$). Then, $X_{n}$ in our case is the union of all Markov cells of the partition of the depth $n$ that do not correspond to the hole. 
Then, one can see that the set $X=\liminf X_{n}$ is exactly the Rauzy gasket. 
 
Now, in order to deduce Theorem \ref{thm:main} from Theorem \ref{AD} we only have to check that $\delta$ > 0. It follows directly from the fact that the size of simplices of the partition decreases exponentially fast because map $T$ is uniformly expanding (see \cite{MN} for details). This completes the proof of Theorem \ref{thm:main}.

\section{Multidimensional case}\label{MD}
Our result also can be generalized for the estimation of the Hausdorff dimension of the Rauzy gasket in any dimension $d>2$ (see \cite{AS} and \cite{DL} for definitions). Indeed, the key ingredients of the proof were the following: 
\begin{itemize}
\item the Rauzy graph is connected, and any complete path is positive; one can check that this statement actually holds in any dimension;
\item the Markov map is uniformly expanding; the proof for larger dimension is the same but the numerical estimates have to be modified; 
\item the Kerckhoff lemma: the proof is exactly the same since it is an adaptation of the statement for IET that was proved for any number of intervals; 
\item more subtle distortion estimates: the idea of the proof is the same but one has to remake the calculations since the combinatorics is slightly different. However, all the statements hold without any changes;
\item fast decaying property and the proof of the main theorem: this part does not require any changes. 
\end{itemize}

\end{document}